\documentclass{cmslatex}
\usepackage[paperwidth=7in, paperheight=10in, margin=.875in]{geometry}
 \usepackage[backref,colorlinks,linkcolor=red,anchorcolor=green,citecolor=blue]{hyperref}
\usepackage{amsfonts,amssymb}
\usepackage{amsmath}
\usepackage{graphicx}
\usepackage{cite}
\usepackage{enumerate}
\usepackage{multirow,makecell}
\renewcommand{\theequation}{\arabic{section}.\arabic{equation}}

   \allowdisplaybreaks
\begin{document}
 \title{Error estimate of the nonuniform $L1$ type formula for the time fractional diffusion-wave equation}


          \author{Hong Sun\thanks{Department of Mathematics and Physics,
Nanjing Institute of Technology, Nanjing 211167; School of Mathematics, Southeast University, Nanjing 210096, P. R. China (sunhongzhal@126.com).}
          \and Yanping Chen\thanks{Corresponding author. School of Mathematical Sciences, South China Normal University, Guangzhou 510631, Guangdong, P. R. China (yanpingchen@scnu.edu.cn).}
          \and Xuan Zhao\thanks{School of Mathematics, Southeast University, Nanjing 210096, P. R. China (xuanzhao11@seu.edu.cn).}}

         \pagestyle{myheadings} \markboth{Error estimate of the nonuniform $L1$ type formula}{H. Sun, Y.P.Chen and X. Zhao} \maketitle
\begin{abstract}
 In this paper, a temporal nonuniform $L1$ type difference scheme is built up for the time fractional diffusion-wave equation with the help of the order reduction technique. The unconditional convergence of the nonuniform difference scheme is proved rigorously in $L^2$ norm. Our main tool is the discrete complementary convolution kernels with respect to the coefficient kernels of the L1 type formula. The positive definiteness of the  complementary convolution kernels is shown to be vital to the stability and convergence. To the best of our knowledge, this property is proved at the first time on the nonuniform time meshes. Two numerical experiments are presented to verify the accuracy and the efficiency of the proposed numerical methods.

          \end{abstract}


\section{Introduction}\label{intro}
In recent years, various phenomena in science and engineering are modeled by the fractional differential equations (FDEs), which in some cases are selected to describe the memorial and the hereditary properties of many viscoelastic materials, economics and dynamics of interfaces between nanoparticle and substrates\cite{Mainardi, Met, Berk, zhou}. Much attention has been placed on the development and {the} research of fractional differential equations \cite{G-G,S-R,Sun-Chen-2014}. However, finding the analytical solutions is difficult for most FDEs due to the fundamental nonlocal property of the fractional derivatives. Existing works are devoted to develop the efficient numerical algorithms for FDEs considering even the singularity of the problem.

Gr\"{u}nwald-Letnikov (GL) formula is a nature way used to compute the fractional derivative in the beginning. Oldham and Spanier developed the first-order GL formula  based on the Gr\"unwald-Letnikov derivative \cite{Old1974} to numerically solve the FDEs. Yuste et al.\cite{Yuste1,Yuste2} presented the explicit and weighted averaged finite difference schemes using the shifted GL formula for the time fractional diffusion-wave equation. In order to improve the accuracy of the approximation, a weighted and shifted Gr\"{u}nwald difference (WSGD) operator with second-order accuracy was presented for solving space fractional diffusion equations in \cite{Tian-Zhou}. Consequently, the WSGD formula is applied to solve the time fractional sub-diffusion or diffusion-wave equation \cite{Wang-Vong,Dimitrov1}.

Besides, some commonly used numerical methods with higher accuracy to approximate the Caputo derivative are derived from the idea of the interpolations. A widely used method with $2-\alpha$ order accuracy, called $L1$ formula in \cite{Sun-Wu, Lin-Xu}, was proposed by using the linear interpolation for the time fractional sub-diffusion equation. By combining $L1$ formula with order reduction method, the authors \cite{Sun-Wu} presented a difference scheme to approximate the time fractional diffusion-wave equation.
Li et.al \cite{Li-Zhao} applied the center difference scheme to approximate the second order derivative in the Caputo derivative of order $\beta\in [1,2].$In \cite{Gao-Sun}, the authors constructed the $L1$-2 formula to discretize the Caputo derivative with order $O(\tau^{3-\alpha})$ by using the quadratics interpolation on all intervals except the linear interpolation on the interval $[t_0, t_1].$
Alikhanov \cite{Alik} proposed the $L2$-$1_\sigma$ formula with $3-\alpha$ order accuracy at the point $t=t_{n+\sigma} (\sigma=1-\frac{\alpha}{2}).$ The formula applied the quadratics interpolation on all intervals while the linear interpolation on the last interval $[t_n, t_{n+\sigma}]$.
In \cite{Sun-Sun}, by using order of reduction, Sun et al. applied the $L2$-$1_\sigma$ formula to discrete the time fractional diffusion-wave equation.
Xu et al.\cite{Xu} presented a high-order finite difference scheme to approximate the Caputo fractional derivative by applying the quadric interpolation polynomial on all intervals. The scheme achieves $(3-\alpha)$-order accuracy in time.

It is noteworthy that the existing numerical analysis in the above proposed numerical algorithms, especially for high order formulas, are valid under the smooth solution hypothesis.
The error estimate based on the $L1$ formula or $L2$-$1_\sigma$ requires the solution of the time fractional differential equation to belong to
$C^2[0, T]$ or $C^3[0,T]$ in time. However, the weak singularity of the fractional Caputo operator at $t=0$ cannot be avoided in the time fractional differential equation, which
 implies that the regularity condition of the solution is restrictive even for the homogeneous problem with a smooth initial data \cite{Sakamoto,Jin2016,Jin2016Siam,Stynes2016,Stynes2017}. The low regularity brings low-accuracy numerical solutions and large computational cost on the temporal uniform mesh.

 Furthermore, the initial singularity has been taken into consideration by many researchers for the time fractional equation. In order to overcome the lost of accuracy caused by the low regularity, some efficient methods are implemented on the nonuniform meshes for the time fractional sub-diffusion equation. Stynes et al. \cite{Stynes2017} presented a difference scheme by using $L1$ formula on the graded meshes for the time fractional diffusion equation. The stability and convergence of the scheme are proved rigorously. They demonstrated that the regularity of the solution and the grading of the mesh affected the order of convergence of the difference scheme. Consequently, there are many numerical algorithms for the time fractional diffusion equation based on the $L1$ formula on the graded meshes, such as the finite element method \cite{Huang-Chen},
 the spectral method \cite{Ren-Chen}, the discontinuous Galekin method \cite{Ren-Huang}, the fast finite difference method \cite{Shen-Sun}. In \cite{Chen1}, the Caputo derivative was approximated by $L2$-$1_\sigma$ formula on the graded mesh. Then a finite difference method was presented for the time-fractional sub-diffusion equation. Under reasonable assumption, the temporal convergence order of the fully scheme is proved to be $ O(N^{-\min\{r\alpha,2\}}).$ There are also existing works devoted to the construction of the numerical schemes on the nonuniform meshes and the adaptive meshes for the sub-diffusion equations\cite{Yuste3,Liao-Li,Roul,Jannelli}.

Whereas, there are relatively few literatures on approximating the time fractional diffusion-wave equation compared to that for the sub-diffusion equation. Shen et al. \cite{Shen-Stynes-Sun} presented the modified $L1$ approximation
by virtue of the order reduction on the graded meshes for the time fractional diffusion-wave equation. The stability and convergence of the scheme
are also analysed under an assumption condition.
 Lyu and Vong \cite{Lyu-Vong1} constructed a temporal nonuniform $L2$ formula (the same as the modified $L1$ formula above)
for the Caputo derivative of order $\beta~(1<\beta<2)$. Based on this formula, a linearized difference scheme was presented for the time-fractional
Benjamin-Bona-Mahony-type equation by the mathematical induction.
In \cite{Lyu-Vong2}, a symmetric fractional order reduction method was introduced to construct $L1$ scheme and $L2$-$1_\sigma$ scheme on the nonuniform temporal meshes for the semilinear fractional diffusion-wave equations, respectively. By use of the mathematical induction method, the convergence is obtained by $H^2$ energy method.

 It is noted that the analysis on the convergence of the proposed scheme is a challenge
for the time fractional diffusion-wave equation. In this paper, we construct the temporal nonuniform difference scheme \cite{Shen-Stynes-Sun}
by combining the order of reduction with the modified $L1$ formula for the time fractional diffusion-wave. The aim of this paper is to present a method different from \cite{Shen-Stynes-Sun,Lyu-Vong1,Lyu-Vong2} for the convergence of the difference
scheme. It relies on a useful discrete tool: the discrete complementary
 convolution (DCC) kernels \cite{Liao-Li} generated by the discrete convolution kernels of the modified $L1$ formula. The properties of the DCC kernels plays important roles in proving the convergence of the difference scheme. With the help of the verified positive definiteness of the DCC kernels, we demonstrate the $L^2$ error estimate of the difference scheme.

 We consider the two dimensional fractional diffusion-wave equation in the following form
\begin{align}
&{}_0^CD_t^\beta u(\mathrm{x},t)=\Delta u(\mathrm{x},t)+f({\rm x},t),~~{\rm x}\in\Omega,~t\in(0, T],\label{diffusion-equation}\\
&u({\rm x},0)=\varphi_1({\rm x}),~~u_t({\rm x},0)=\varphi_2({\rm x}),~~{\rm x}\in\bar{\Omega},\label{boundaryvalue}
\end{align}
subjected to the periodic boundary conditions, where $\Omega=(0,L)^2\subset \mathcal{R}^2$ and
${}_0^CD_t^\beta u(t),$ $(1<\beta<2)$  denotes Caputo
fractional derivative of order $\beta$ defined by
$$ {}_0^CD_t^\beta u(t)=\int_0^t
 \omega_{2-\beta}(t-s)u''(s)ds, ~~{\rm where}~~ \omega_\gamma(t)=\frac{t^{\gamma-1}}{\Gamma(\gamma)}.$$

The rest of the paper is arranged as follows. In Section \ref{Sec.2}, the discrete $L1$ type formula on the nonuniform meshes is presented for the Caputo derivative of order $0<\alpha<1$. The useful properties of the coefficient kernels of the proposed scheme are discussed. The DCC kernels are introduced with some proved properties. Section \ref{Sec.3} is devoted to constructing the temporal nonuniform difference scheme. The unconditional convergence of
the scheme in $L^2$ norm is proved rigorously in Section \ref{Sec.4}.
 In Section \ref{Sec.5}, two numerical examples are provided to verify the theoretical results.
The paper ends with a brief conclusion in Section \ref{Sec.6}.

\section{The discrete formula of the Caputo derivative}\label{Sec.2}

 For the nonuniform time levels $0=t_0<t_1<t_2<\cdots<t_N=T,$ we denote $\tau_n=t_n-t_{n-1}$
as the $n$th step size and $\tau_{n-\frac{1}{2}}=\frac{\tau_n+\tau_{n-1}}{2}$ for $n>1$, $\tau_{\frac{1}{2}}=\frac{\tau_1}{2}.$
 Denote $t_{n-\frac{1}{2}}=t_{n-1}+\frac{\tau_n}{2}$ for $n>1$
and $t_{-\frac{1}{2}}=t_0.$
Let $v^k\approx v(t_n),~~\nabla_\tau v^n=v^n-v^{n-1},~~\delta_tv^{n-\frac{1}{2}}=\nabla_\tau v^n/\tau_n$ and
$v^{n-\frac{1}{2}}=\frac{1}{2}(v^n+v^{n-1}).$
In this paper, the basic assumption on the nonuniform time meshes is as follows
\begin{align}
\tau_{n-1}\le \tau_n,~~2\le n\le N.\label{step-condition}
\end{align}
 Now, we present the approximation formula for the Caputo derivative of order $\alpha~(0< \alpha\le1)$ at the point $t_{n-\frac{1}{2}}.$
 Let $\Pi_{1,k} v(t)$ be linear interpolation of $v(t)$ over the small interval $[t_{k-\frac{3}{2}}, t_{k-\frac{1}{2}}] $ for $1\le k\le n.$ Then we get
 \begin{align}
{}_0^CD_t^\alpha v(t_{n-\frac{1}{2}})=&\sum_{k=1}^{n}\int_{t_{k-\frac{3}{2}}}^{t_{k-\frac{1}{2}}}\omega_{1-\alpha}(t_{n-\frac{1}{2}}-s)v'(s)ds\nonumber\\
\approx&
\sum_{k=1}^{n}\frac{1}{\tau_{k-\frac{1}{2}}}\int_{t_{k-\frac{3}{2}}}^{t_{k-\frac{1}{2}}}
\omega_{1-\alpha}(t_{n-\frac{1}{2}}-s)(v^{k-\frac{1}{2}}-v^{k-\frac{3}{2}})ds\nonumber\\
=&\sum_{k=1}^{n}a_{n-k}^{(n)}\nabla_\tau v^{k-\frac{1}{2}}:=\mathcal{D}_{\tau}^\alpha v^{n-\frac{1}{2}}, ~~n\ge1,\label{appr-formula}
\end{align}
 where $\nabla_\tau v^{\frac{1}{2}}=\frac{1}{2}(v^{1}-v^0)$ and
\begin{align}
a_{n-k}^{(n)}=\frac{1}{\tau_{k-\frac{1}{2}}}\int_{t_{k-\frac{3}{2}}}^{t_{k-\frac{1}{2}}}
\omega_{1-\alpha}(t_{n-\frac{1}{2}}-s)ds,~1\le k\le n.\label{coefficient}
\end{align}
The properties of the discrete coefficient kernels $a_{n-k}^{(n)}$ are stated in the following lemma. These properties are necessary to prove
the theoretical results of the numerical methods for the fractional diffusion-wave equation.
\begin{lemma}\label{property-a}
For any $\alpha$ and $\{a_{n-k}^{(n)}\}$ defined in (\ref{coefficient}), it holds that \\
{\rm (I)} ~~$a_k^{(n)}>0,~~0\le k\le n-1, ~~a_{k}^{(n)}\le a_{k-1}^{(n)},~~1\le k\le n-1,$\\
{\rm (II)}~~$a_{k}^{(n)}\le a_{k-1}^{(n-1)},~~a_{k-1}^{(n-1)}a_{k+1}^{(n)}\ge a_k^{(n-1)}a_k^{(n)},~~1\le k\le n-2,$\\
{\rm(III)}~~$a_{k}^{(n)}<a_{k}^{(n-1)},~~0\le k\le n-2,$\\
{\rm (IV)}~~$\frac{a_0^{(k)}}{a_{k-2}^{(k)}}<\frac{(t_{k-\frac{1}{2}}-t_{\frac{1}{2}})^\alpha}{(1-\alpha)(\tau_{k-\frac{1}{2}})^\alpha},~~2\le k\le n,$\\
{\rm (V)}~~$0<\omega_{1-\alpha}(t_{n-\frac{1}{2}}-t_{k-\frac{1}{2}})-\omega_{1-\alpha}(t_{n-\frac{1}{2}}-t_{k-\frac{3}{2}})
\le a_{n-k-1}^{(n)}-a_{n-k}^{(n)},~~2\le k\le n-1.$\\
\end{lemma}
\begin{proof}(I) Making use of the integral mean value theorem, there exists constant $\xi_k$ such that
\begin{align*}
a_k^{(n)}=\frac{1}{\tau_{n-k-\frac{1}{2}}}\int_{t_{n-k-\frac{3}{2}}}^{t_{n-k-\frac{1}{2}}}& \omega_{1-\alpha}(t_{n-\frac{1}{2}}-s)ds
 =\omega_{1-\alpha}(t_{n-\frac{1}{2}}-\xi_k),
\end{align*}
where $\xi_k\in(t_{n-k-\frac{3}{2}}, t_{n-k-\frac{1}{2}}).$
Noticing the monotonic decreasing of $\omega_{1-\alpha}(s)$ and
 $\omega_{1-\alpha}(s)>0,$ one can get the inequality in (I).

(II) We introduce the following two auxiliary sequences
$$\psi_{n-k}^{(n)}=\frac{a_{n-k}^{(n)}}{a_{n-1-k}^{(n-1)}},~~
b_{n,k}(\theta)=\frac{1}{\tau_{k-\frac{1}{2}}}\int_{t_{k-\frac{3}{2}}}^{t_{k-\frac{3}{2}}
+\theta\tau_{k-\frac{1}{2}}}\omega_{1-\alpha}(t_{n-\frac{1}{2}}-s) ds~~{\rm for}~~1\le k\le n-1.$$
Applying the Cauchy mean-value theorem, there exists $\zeta_k\in (0,1)$ such that
\begin{align*}
  \psi_{n-k}^{(n)}&=\frac{b_{n,k}(1)-b_{n,k}(0)}{b_{n-1,k}(1)-b_{n-1,k}(0)}=\frac{b'_{n,k}(\zeta_k)}{b'_{n-1,k}(\zeta_k)}
  =\frac{\omega_{1-\alpha}(t_{n-\frac{1}{2}}-t_{k-\frac{3}{2}}-\zeta_k\tau_{k-\frac{1}{2}})}
  {\omega_{1-\alpha}(t_{n-\frac{3}{2}}-t_{k-\frac{3}{2}}-\zeta_k\tau_{k-\frac{1}{2}})}\\
  &=\left(\frac{t_{n-\frac{3}{2}}-t_{k-\frac{3}{2}}-\zeta_k\tau_{k-\frac{1}{2}}}
  {t_{n-\frac{1}{2}}-t_{k-\frac{3}{2}}-\zeta_k\tau_{k-\frac{1}{2}}}\right)^\alpha,~~1\le k\le n-1.
\end{align*}
It follows that
\begin{align*}
0\le \psi_{n-k}^{(n)}\le 1,~~1\le k\le n-1,
\end{align*}
which implies the first inequality is valid in (II).
Noticing that $y=\frac{t_{n-\frac{3}{2}}-x}{t_{n-\frac{1}{2}}-x}$ is decreasing with respect to $x>0,$ it yields
$$\psi_{k-1}^{(n)}<\psi_k^{(n)},~~1\le k\le n-1,~~n\ge 2.$$ Then we get the second inequality in (II).

(III) By the definition of $a_{n-k}^{(n)}$ and the variable substitution $t=\frac{s-t_{k-\frac{3}{2}}}{\tau_{k-\frac{1}{2}}}$, it yields
\begin{align*}
a_{n-k}^{(n)}=\int_0^1\omega_{1-\alpha}\Big(t_{n-\frac{1}{2}}-t_{k-\frac{3}{2}}-\tau_{k-\frac{1}{2}}t\Big)dt.
\end{align*}
With the help of the differential mean-value theorem, there exists $\eta_j\in (0,1)$ such that
\begin{align*}
&a_{n-k}^{(n)}-a_{n-k}^{(n-1)}\\
=&\frac{1}{\Gamma(1-\alpha)}
\int_0^1\Big[\Big(t_{n-\frac{1}{2}}-t_{k-\frac{3}{2}}-\tau_{k-\frac{1}{2}}t\Big)^{-\alpha}
-\Big(t_{n-\frac{3}{2}}-t_{k-\frac{5}{2}}-\tau_{k-\frac{3}{2}}t\Big)^{-\alpha}\Big]dt\\
=&\frac{\alpha}{\Gamma(1-\alpha)}
\int_0^1(t_{n-\frac{3}{2}}-\eta_j)^{-\alpha-1}\Big(\tau_{k-\frac{1}{2}}t+\tau_{k-\frac{3}{2}}(1-t)-\tau_{n-\frac{1}{2}}\Big)dt.
\end{align*}
The condition $\tau_{k-1}\le\tau_{k},~~(2\le k\le n)$ implies that
\begin{align*}
\tau_{k-\frac{1}{2}}t+\tau_{k-\frac{3}{2}}(1-t)-\tau_{n-\frac{1}{2}}\le 0.
\end{align*}
Consequently, we obtain the desired inequality (III).

(IV) It follows from (\ref{coefficient}) that
\begin{align*}
a_0^{(k)}=\frac{1}{\tau_{k-\frac{1}{2}}}\int_{t_{k-\frac{3}{2}}}^{t_{k-\frac{1}{2}}}\omega_{1-\alpha}(t_{k-\frac{1}{2}}-s) ds
=\frac{1}{\tau_{k-\frac{1}{2}}}\omega_{2-\alpha}(\tau_{k-\frac{1}{2}})
\end{align*}
and
\begin{align*}
a_{k-2}^{(k)}=\frac{1}{\tau_{\frac{3}{2}}}\int_{t_{\frac{1}{2}}}^{t_{\frac{3}{2}}}\omega_{1-\alpha}(t_{k-\frac{1}{2}}-s) ds
\ge\omega_{1-\alpha}(t_{k-\frac{1}{2}}-t_\frac{1}{2}).
\end{align*}
Then, it yields
\begin{align*}
\frac{a_{0}^{(k)}}{a_{k-2}^{(k)}}\le \frac{\omega_{2-\alpha}
(\tau_{k-\frac{1}{2}})}{\tau_{k-\frac{1}{2}}\omega_{1-\alpha}(t_{k-\frac{1}{2}}-t_{\frac{1}{2}})}\le
\frac{(t_{k-\frac{1}{2}}-t_{\frac{1}{2}})^\alpha}{(1-\alpha)(\tau_{k-\frac{1}{2}})^\alpha}.
\end{align*}

(V) Exchanging the order of integration, it arrives at
\begin{align*}
a_{n-k}^{(n)}-\omega_{1-\alpha}(t_{n-\frac{1}{2}}-t_{k-\frac{3}{2}})
=&\frac{1}{\tau_{k-\frac{1}{2}}}\int_{t_{k-\frac{3}{2}}}^{t_{k-\frac{1}{2}}}
\big[\omega_{1-\alpha}(t_{n-\frac{1}{2}}-s)-\omega_{1-\alpha}(t_{n-\frac{1}{2}}-t_{k-\frac{3}{2}})\big] ds\\
=&-\frac{1}{\tau_{k-\frac{1}{2}}}\int_{t_{k-\frac{3}{2}}}^{t_{k-\frac{1}{2}}}
\Big(\int_{t_{k-\frac{3}{2}}}^s\omega_{-\alpha}(t_{n-\frac{1}{2}}-\mu)d\mu\Big)ds\\
=&\int_{t_{k-\frac{3}{2}}}^{t_{k-\frac{1}{2}}}\frac{\mu-t_{k-\frac{1}{2}}}{\tau_{k-\frac{1}{2}}} \omega_{-\alpha}(t_{n-\frac{1}{2}}-\mu)d\mu.
\end{align*}
Define the following auxiliary function
$$c_k(\theta)=\int_{t_{k-\frac{3}{2}}}^{t_{k-\frac{3}{2}}+\theta\tau_{k-\frac{1}{2}}}
\frac{\mu-t_{k-\frac{3}{2}}-\theta\tau_{k-\frac{1}{2}}}{\tau_{k-\frac{1}{2}}}\omega_{-\alpha}(t_{n-\frac{1}{2}}-\mu)d\mu,~~1\le k\le n,$$
it is easy to check $c_k(0)=c_k'(0),~1\le k\le n.$ By virtue of the Cauchy mean-value theorem,
there exists $\gamma_k,~\rho_k\in (0,1)$ such that
\begin{align*}
&\frac{a_{n-k}^{(n)}-\omega_{1-\alpha}(t_{n-\frac{1}{2}}-t_{k-\frac{3}{2}})}
{a_{n-k-1}^{(n)}-\omega_{1-\alpha}(t_{n-\frac{1}{2}}-t_{k-\frac{1}{2}})}\\
=&\frac{c_k(1)-c_k(0)}{c_{k+1}(1)-c_{k+1}(0)}=\frac{c'_k(\gamma_k)}{c'_{k+1}(\gamma_k)}
=\frac{c'_k(\gamma_k)-c'_k(0)}{c'_{k+1}(\gamma_k)-c'_{k+1}(0)}
=\frac{c''_k(\rho_k)}{c''_{k+1}(\rho_k)}\\
=&\frac{\tau_{k-\frac{1}{2}}\omega_{-\alpha}(t_{n-\frac{1}{2}}-t_{k-\frac{3}{2}}-\rho_k\tau_{k-\frac{1}{2}})}
{\tau_{k+\frac{1}{2}}\omega_{-\alpha}(t_{n-\frac{1}{2}}-t_{k-\frac{1}{2}}-\rho_k\tau_{k+\frac{1}{2}})}
\le\Big(\frac{t_{n-\frac{1}{2}}-t_{k-\frac{1}{2}}-\rho_k\tau_{k+\frac{1}{2}}}
{t_{n-\frac{1}{2}}-t_{k-\frac{3}{2}}-\rho_k\tau_{k-\frac{1}{2}}}\Big)^{1+\alpha}\\
\le&\Big(\frac{t_{n-\frac{1}{2}}-t_{k-\frac{1}{2}}}
{t_{n-\frac{1}{2}}-t_{k-\frac{1}{2}}+(1-\rho_k)\tau_{k-\frac{1}{2}}}\Big)^{1+\alpha}\le1.
\end{align*}
The proof ends.
\end{proof}

\subsection{The properties of the DCC kernels}
In this section, we present the DCC kernels generated by the discrete convolution kernels $a_{n-k}^{(n)}$ (proposed in \cite{Liao-Li,Liao-Tang}).
The DCC kernels are the key to prove the convergence of the difference scheme (\ref{discrete-scheme})-(\ref{discrete-u-v-initialvalue}). The discrete tool DCC kernels $p_{n-k}^{(n)}$ are defined by
\begin{align}
p_0^{(n)}=\frac{1}{a_0^{(n)}},~~p_{n-k}^{(n)}=\frac{1}{a_0^{(k)}}\sum_{j=k+1}^{n}(a_{j-k-1}^{(j)}-a_{j-k}^{(j)})p_{n-j}^{(n)},~~1\le k\le n-1.\label{DCC kernel}
\end{align}
It is equivalent to the following identity
\begin{align}
\sum_{j=k}^{n}p_{n-j}^{(n)}a_{j-k}^{(j)}\equiv1,~~1\le k\le n.\label{DCC equlity}
\end{align}

The following lemma presents the linear interpolation error
formula with an integral remainder.
\begin{lemma}\label{interpolation-error}
Assume $q\in C^2(0,T])$ and let $\Pi_{1,k} q(t)$ be linear interpolation of $q(t)$ over the small interval $[t_{k-\frac{3}{2}}, t_{k-\frac{1}{2}}] $ for $1\le k\le n,$ then, the linear interpolation error gives
\begin{align*}
q(t)-\Pi_{1,k}q(t)=
\int_{t_{k-\frac{3}{2}}}^{t_{k-\frac{1}{2}}}\chi_k(t,\lambda)q''(\lambda) d\lambda,~~t\in[t_{k-\frac{3}{2}}, t_{k-\frac{1}{2}}],~~1\le k\le n,
\end{align*}
where the Peano kernel $\chi_k(t,\lambda)=\max\{t-\lambda, 0\}-\frac{t-t_{k-\frac{3}{2}}}{\tau_{k-\frac{1}{2}}}(t_{k-\frac{1}{2}}-\lambda)$ such that
$$-\frac{(t_{k-\frac{1}{2}}-\lambda)}{\tau_{k-\frac{1}{2}}}(t-t_{k-\frac{3}{2}})\le \chi_k(t,\lambda)\le 0, ~~t, \lambda\in[t_{k-\frac{3}{2}}, t_{k-\frac{1}{2}}].$$
\end{lemma}
\begin{proof}
By similar process of Lemma 3.1 in \cite{Liao-Li}, it is easy to obtain the result.
\end{proof}

The properties of the DCC kernels $p_{n-k}^{(n)}$ are demonstrated in the following two lemmas.
\begin{lemma}\label{DCC-property}
The DCC kernels $p_{n-k}^{(n)}$ are non-negative, i.e.,
$$p_{n-k}^{(n)}\ge 0,~~1\le k\le n.$$
Moreover, the DCC kernels $p_{n-k}^{(n)}$ satisfies the following property
$$\sum_{j=1}^{n}p_{n-j}^{(n)}\le \omega_{1+\alpha}(t_{n-\frac{1}{2}}).$$
\end{lemma}
\begin{proof}
Noticing the property $a_{j-1}^{(n)}\ge a_{j}^{(n)}$ in Lemma \ref{property-a}, we get $p_{n-k}^{(n)}\ge 0$.

We approximate the Caputo derivative of the function $\omega_{1+\alpha}(t)$ by the formula (\ref{appr-formula}) at $t=t_{j-\frac{1}{2}}$ and let $S^j$
be the truncation error. We have
\begin{align}
S^j=&{}_0^CD_t^\alpha \omega_{1+\alpha}(t_{j-\frac{1}{2}})-\sum_{k=1}^{j}a_{j-k}^{(j)}\nabla_\tau \omega_{1+\alpha}(t_{k-\frac{1}{2}})\nonumber\\
=&\sum_{k=1}^j\int_{t_{k-\frac{3}{2}}}^{t_{k-\frac{1}{2}}}\omega_{1-\alpha}(t_{j-\frac{1}{2}}-s)\Big(\omega_{1+\alpha}(s)-\Pi_{1,k}\omega_{1+\alpha}(s)\Big)' ds\nonumber\\=:&\sum_{k=1}^jS_k^j,\label{w_{1+alpha}appr}
\end{align}
then, by applying the integration by parts and Lemma \ref{interpolation-error}, one arrives at
\begin{align}
S_k^j&=\int_{t_{k-\frac{3}{2}}}^{t_{k-\frac{1}{2}}}\omega_{-\alpha}(t_{j-\frac{1}{2}}-s)\Big(\omega_{1+\alpha}(s)-\Pi_{1,k}\omega_{1+\alpha}(s)\Big) ds\nonumber\\
&=\int_{t_{k-\frac{3}{2}}}^{t_{k-\frac{1}{2}}}\omega_{-\alpha}(t_{j-\frac{1}{2}}-s) ds
\int_{t_{k-\frac{3}{2}}}^{t_{k-\frac{1}{2}}}\chi_k(s,\lambda)\omega_{1+\alpha}''(\lambda) d\lambda\le 0,\label{trunca-w}
\end{align}
where we used the fact that $\omega_{-\alpha}(t_{j-\frac{1}{2}}-s)<0,~~\chi_k(s,\lambda)\le0$ and $\omega_{1+\alpha}''(\lambda)\le0$
in the last step. Noticing that ${}_0^CD_t^\alpha \omega_{1+\alpha}(t)=1,$
it follows from (\ref{w_{1+alpha}appr}) and (\ref{trunca-w}) that
\begin{align}
1-\sum_{k=1}^{j}a_{j-k}^{(j)}\nabla_\tau \omega_{1+\alpha}(t_{k-\frac{1}{2}})\le0.\label{**}
\end{align}
Multiplying the inequality (\ref{**}) by $p_{n-j}^{(n)}$ and summing up $j$ from 1 to $n,$ it yields
$$\sum_{j=1}^np_{n-j}^{(n)}\le \sum_{j=1}^np_{n-j}^{(n)}\sum_{k=1}^{j}a_{j-k}^{(j)}\nabla_\tau \omega_{1+\alpha}(t_{k-\frac{1}{2}})=\omega_{1+\alpha}(t_{n-\frac{1}{2}}).$$
This completes the proof.
\end{proof}
\begin{lemma}{\rm\cite{Liao-Tang}}\label{positive-property}
Assume that the two kernels $\xi_{n-k}^{(n)}$ and $\eta_{n-k}^{(n)}$ satisfying the following orthogonal identity
  \begin{align*}
\sum_{k=j}^{n}\xi_{n-k}^{(n)}\eta_{k-j}^{(k)}=\delta_{nj},~~1\le j\le n.
\end{align*}
Then $\xi_{n-k}^{(n)}$ are positive definite if and only if $\eta_{n-k}^{(n)}$ are positive definite.
\end{lemma}
With the help of Lemma \ref{positive-property}, we obtain the following lemma which plays an important role in proving the convergence
of the numerical scheme for the time fractional diffusion-wave equation.
\begin{lemma}\label{DCCposi}
The DCC kernels $p_{n-k}^{(n)}$ defined by (\ref{DCC kernel}) are positive definite.
\end{lemma}
\begin{proof} Denote
\begin{align*}
\zeta_{k-j}^{(k)}=\left\{
\begin{aligned}
&a_0^{(k)},~~j=k,\\
&a_{k-j}^{(k)}-a_{k-j-1}^{(k)},~~j\neq k,
\end{aligned}\right.
\end{align*}
the equality (\ref{DCC kernel}) is rewritten as
\begin{align*}
\sum_{k=j}^{n}p_{n-k}^{(n)}\zeta_{k-j}^{(k)}=\delta_{nj},~~1\le j\le n.
\end{align*}
By virtue of Lemma \ref{positive-property}, we only need to prove that $\zeta_{k-j}^{(k)}$ is positive definite.

Next, we prove the positive definiteness of the kernels $\zeta_{k-j}^{(k)}$.
For any real sequence $\{w_k\}_{k=1}^n,$ with the help of Young's inequality, it holds that
\begin{align*}
\sum_{k=1}^nw_k\sum_{j=1}^k\zeta_{k-j}^{(k)}w_j&=\sum_{k=1}^nw_k\Big(\zeta_0^{(k)}w_k+\sum_{j=1}^{k-1}\zeta_{k-j}^{(k)}w_j\Big)\\
&=\sum_{k=1}^na_0^{(k)}w_k^2+\sum_{k=1}^nw_k\sum_{j=1}^{k-1}(a_{k-j}^{(k)}-a_{k-j-1}^{(k)})w_j\\
&\ge\sum_{k=1}^na_0^{(k)}w_k^2-\frac{1}{2}\sum_{k=1}^nw_k^2\sum_{j=1}^{k-1}(a_{k-j-1}^{(k)}-a_{k-j}^{(k)})\\
&~-\frac{1}{2}\sum_{k=1}^n\sum_{j=1}^{k-1}(a_{k-j-1}^{(k)}-a_{k-j}^{(k)})w_j^2\\
&=\frac{1}{2}\sum_{k=1}^n(a_0^{(k)}+a_{k-1}^{(k)})w_k^2-\frac{1}{2}\sum_{j=1}^{n-1}w_j^2\sum_{k=j+1}^{n}(a_{k-j-1}^{(k)}-a_{k-j}^{(k)}).
\end{align*}
For the second term on the right hand of the above inequality, using the property (III) in Lemma \ref{property-a}, we have
\begin{align*}
\sum_{j=1}^{n-1}w_j^2\sum_{k=j+1}^{n}(a_{k-j-1}^{(k)}-a_{k-j}^{(k)})
=&\sum_{j=1}^{n-1}w_j^2\Big[a_0^{(j+1)}+\sum_{k=j+1}^{n-1}(a_{k-j}^{(k+1)}-a_{k-j}^{(k)})-a_{n-j}^{(n)}\Big]\\
\le&\sum_{j=1}^{n-1}a_0^{(j+1)}w_j^2.
\end{align*}
Then, it follows from the fact $a_0^{(k)}\le a_0^{(k-1)}$ in Lemma \ref{property-a} (III) that
 \begin{align*}
\sum_{k=1}^nw_k\sum_{j=1}^k\zeta_{k-j}^{(k)}w_j\ge &\frac{1}{2}\sum_{k=1}^n(a_0^{(k)}+a_{k-1}^{(k)})w_k^2
-\frac{1}{2}\sum_{j=1}^{n-1}a_0^{(j+1)}w_j^2\\
\ge&\frac{1}{2}\sum_{k=1}^{n-1}\Big(a_0^{(k)}-a_0^{(k+1)}\Big)w_k^2\ge0,
\end{align*}
which implies $\zeta_{k-j}^{(k)}$ is positive definite.  The proof ends.
\end{proof}

We denote the local consistency error of the formula (\ref{appr-formula}) at the time $t_{n-\frac{1}{2}}$ by $$R^{n}={}_0^CD_t^\alpha v(t_{n-\frac{1}{2}})-\mathcal{D}_{\tau}^\alpha v^{n-\frac{1}{2}}:=\sum_{k=1}^nR_{k}^n,~~n\ge 1,$$
where by exchanging the order of integration, one arrives at
\begin{align}
R_{k}^n=&-\frac{1}{\tau_{k-\frac{1}{2}}}\int_{t_{k-\frac{3}{2}}}^{t_{k-\frac{1}{2}}}\omega_{1-\alpha}(t_{n-\frac{1}{2}}-s)
\Big[\int_s^{t_{k-\frac{1}{2}}}v''(t)(t_{k-\frac{1}{2}}-t) dt\nonumber\\
&-\int_s^{t_{k-\frac{3}{2}}}v''(t)(t_{k-\frac{3}{2}}-t) dt\Big] ds\nonumber\\
=&-\frac{1}{\Gamma(2-\alpha)}\int_{t_{k-\frac{3}{2}}}^{t_{k-\frac{1}{2}}}
\Big[\frac{t_{k-\frac{1}{2}}-t}{\tau_{k-\frac{1}{2}}}(t_{n-\frac{1}{2}}-t_{k-\frac{3}{2}})^{1-\alpha}
-\frac{t_{k-\frac{3}{2}}-t}{\tau_{k-\frac{1}{2}}}(t_{n-\frac{1}{2}}-t_{k-\frac{1}{2}})^{1-\alpha}\nonumber\\
&~~-(t_{n-\frac{1}{2}}-t)^{1-\alpha}\Big]v''(t) dt\nonumber\\
=&\int_{t_{k-\frac{3}{2}}}^{t_{k-\frac{1}{2}}}
\tilde{\Pi}_{1,k}\omega_{2-\alpha}(t_{n-\frac{1}{2}}-t)v''(t) dt,~~1\le k\le n,\label{truncationerror-n1}
\end{align}
where $\widetilde{\Pi}_{1,k}\omega_{2-\alpha}(t_{n-\frac{1}{2}}-t)$
is the linear interpolation error of the function $\omega_{2-\alpha}(t_{n-\frac{1}{2}}-t)$ on the interval $[t_{k-\frac{3}{2}}, t_{k-\frac{1}{2}}],~~1\le k\le n.$

We are now in the position to estimate the global approximation errors $\sum\limits_{j=1}^np_{n-j}^{(n)}|R^j|.$
\begin{lemma}\label{truncationerror}
Assuming $v\in C^2((0,T])$ with $\int_0^T t|v''(t)| dt<\infty,$
if the nonuniform grid satisfies (\ref{step-condition}), it holds that
$$\sum\limits_{j=1}^np_{n-j}^{(n)}|R^j|\le 2\sum_{j=1}^{n} p_{n-j}^{(n)}a_0^{(j)}\int_{t_{j-\frac{3}{2}}}^{t_{j-\frac{1}{2}}}\Big(t-t_{j-\frac{3}{2}}\Big)|v''(t)| dt.$$
\end{lemma}
\begin{proof} For $1\le k\le n-1,$ by using Lemma \ref{interpolation-error}, it reads
\begin{align}
\widetilde{\Pi}_{1,k}\omega_{2-\alpha}(t_{n-\frac{1}{2}}-t)=&
\int_{t_{k-\frac{3}{2}}}^{t_{k-\frac{1}{2}}}\chi_k(t,\lambda)\omega_{2-\alpha}''(t_{n-\frac{1}{2}}-\lambda) d\lambda\nonumber\\
\le&(t_{k-\frac{3}{2}}-t)\int_{t_{k-\frac{3}{2}}}^{t_{k-\frac{1}{2}}}\omega_{2-\alpha}''(t_{n-\frac{1}{2}}-\lambda) d\lambda\nonumber\\
=&(t-t_{k-\frac{3}{2}})\Big[\omega_{1-\alpha}(t_{n-\frac{1}{2}}-t_{k-\frac{1}{2}})-\omega_{1-\alpha}(t_{n-\frac{1}{2}}-t_{k-\frac{3}{2}})\Big]\nonumber\\
\le&(t-t_{k-\frac{3}{2}})(a_{n-k-1}^{(n)}-a_{n-k}^{(n)}),~~{\rm for}~~ t\in(t_{k-\frac{3}{2}}, t_{k-\frac{1}{2}}).\label{truncation-k}
\end{align}
For $k=n,$ noticing the decreasing of $\omega_{2-\alpha}(t_{n-\frac{1}{2}}-t)$ with respect to $t$, it yields
\begin{align}
0\le \widetilde{\Pi}_{1,n}\omega_{2-\alpha}(t_{n-\frac{1}{2}}-t)\le & \omega_{2-\alpha}(t_{n-\frac{1}{2}}-t_{n-\frac{3}{2}})-\Pi_{1,n}\omega_{2-\alpha}(t_{n-\frac{1}{2}}-t)
=(t-t_{n-\frac{3}{2}})a_0^{(n)}.\label{n-error}
\end{align}
Now, we estimate the truncation errors $R_n^n$ and $R_k^n (1\le k\le n-1),$ respectively. Denote
$$G^k=\int_{t_{k-\frac{3}{2}}}^{t_{k-\frac{1}{2}}}(t-t_{k-\frac{3}{2}})|v''(t)| dt,~1\le k\le n.$$
It follows from (\ref{truncationerror-n1}) and (\ref{n-error}) that
\begin{align}
|R_{n}^n|&\le \int_{t_{n-\frac{3}{2}}}^{t_{n-\frac{1}{2}}}\widetilde{\Pi}_{1,n}\omega_{2-\alpha}(t_{n-\frac{1}{2}}-t)|v''(t)| dt\\
&\le a_0^{(n)}\int_{t_{n-\frac{3}{2}}}^{t_{n-\frac{1}{2}}}(t-t_{n-\frac{3}{2}})|v''(t)| dt=a_0^{(n)}G^n,~~n\ge1.\label{R_n^n}
\end{align}
For $1\le k\le n-1~~(n\ge2), $ with the help of (\ref{truncation-k}), we have
\begin{align}
\sum_{k=1}^{n-1}|R_{k}^n|
\le&\sum_{k=1}^{n-1}\int_{t_{k-\frac{3}{2}}}^{t_{k-\frac{1}{2}}}\widetilde{\Pi}_{1,k}\omega_{2-\alpha}(t_{n-\frac{1}{2}}-t)|v''(t)| dt\nonumber\\
\le&\sum_{k=1}^{n-1}(a_{n-k-1}^{(n)}-a_{n-k}^{(n)})\int_{t_{k-\frac{3}{2}}}^{t_{k-\frac{1}{2}}}(t-t_{k-\frac{3}{2}})|v''(t)| dt\nonumber\\
=&\sum_{k=1}^{n-1}(a_{n-k-1}^{(n)}-a_{n-k}^{(n)})G^k.\label{R_n^k}
\end{align}
Combining the inequality (\ref{R_n^n}) with the inequality (\ref{R_n^k}), the estimate holds
\begin{align}
|R^j|=\sum_{k=1}^j|R_{k}^j|\le\sum_{k=1}^{j-1}(a_{j-k-1}^{(j)}-a_{j-k}^{(j)})G^k+a_0^{(j)}G^j,~~1\le j\le n.\label{truncationerror-nn}
\end{align}
Multiplying (\ref{truncationerror-nn}) by $p_{n-j}^{(n)}$, summing up $j$ from 1 to $n$ and then exchanging the summation order
\begin{align*}
\sum_{j=1}^np_{n-j}^{(n)}|R^j|&\le \sum_{j=2}^np_{n-j}^{(n)}\sum_{k=1}^{j-1}(a_{j-k-1}^{(j)}-a_{j-k}^{(j)})G^k
+\sum_{j=1}^{n}p_{n-j}^{(n)}a_0^{(j)}G^j\\
&=\sum_{k=1}^{n-1}G^kp_{n-k}^{(n)}a_{0}^{(k)}+\sum_{j=1}^{n}G^j p_{n-j}^{(n)}a_0^{(j)}\\
&\le2\sum_{j=1}^{n} p_{n-j}^{(n)}a_0^{(j)}\int_{t_{j-\frac{3}{2}}}^{t_{j-\frac{1}{2}}}\Big(t-t_{j-\frac{3}{2}}\Big)|v''(t)| dt.
\end{align*}
The proof ends.

It follows that the error bound in Lemma \ref{truncationerror} is asymptotically compatible with the truncation error of the backward Euler scheme.
Actually, as the fractional order $\alpha\rightarrow 1,$ it yields $p_{n-j}^{(n)}\rightarrow \tau_j$ and $a_{0}^{(j)}=\frac{1}{\tau_j}$ for $1\le j\le n.$ Then
we arrive at
$$\sum_{j=1}^{n} p_{n-j}^{(n)}a_0^{(j)}\int_{t_{j-\frac{3}{2}}}^{t_{j-\frac{1}{2}}}\Big(t-t_{j-\frac{3}{2}}\Big)|v''(t)| dt
\rightarrow\sum_{j=1}^{n} \int_{t_{j-\frac{3}{2}}}^{t_{j-\frac{1}{2}}}\Big(t-t_{j-\frac{3}{2}}\Big)|v''(t)| dt,$$
which achieves the temporal order $O(\tau).$
However, the following corollary is not asymptotically compatible as the fractional order $\alpha\rightarrow1$ due to the lack of the proper estimates
for the DCC kernels $p_{n-j}^{(n)}.$
\begin{cor}\label{truncation-error2}
Assume $v\in C^2((0,T])$ and there exists a constant $c_v>0$ such that
\begin{align}
|v''(t)|\le c_v(1+t^{\sigma-2}),~~0\le t\le T,\label{v-condition}
\end{align}
where $\sigma\in(0,1)\cup (1,2)$ is a regularity parameter. If the nonuniform grid satisfies (\ref{step-condition}), it holds
$$\sum\limits_{j=1}^np_{n-j}^{(n)}|R^j|\le c_v\Big(\tau_1^\sigma
+\frac{1}{1-\alpha}\max_{2\le j\le n}(t_{j-\frac{1}{2}}-t_{\frac{1}{2}})^\alpha t_{j-\frac{3}{2}}^{\sigma-2}\tau_{j-\frac{1}{2}}^{2-\alpha}\Big),~~n\ge1.$$
\end{cor}
\begin{proof}
Noticing that $|v''(t)|\le c_v(1+t^{\sigma-2})$, we have
$$G^1\le c_v\tau_1^\sigma,~~{\rm and}~~G^k\le
 \int_{t_{k-\frac{3}{2}}}^{t_{k-\frac{1}{2}}}(t-t_{k-\frac{3}{2}})c_vt^{\sigma-2} dt
 \le c_v (\tau_{k-\frac{1}{2}})^2t_{k-\frac{3}{2}}^{\sigma-2},~~2\le k\le n.$$
It follows from $p_{n-k}^{(n)}>0,~~a_{n-k}^{(n)}>0$ and the identity (\ref{DCC kernel}) that $\sum_{j=2}^np_{n-j}^{(n)}a_{j-2}^{(j)}= 1.$ Then making use of Lemma \ref{property-a} (IV), one arrives at
\begin{align*}
\sum_{j=1}^np_{n-j}^{(n)}|R^j|&\le 2p_{n-1}^{(n)}a_0^{(1)}G^1 +2\sum_{j=2}^{n}G^j p_{n-j}^{(n)}a_0^{(j)}\\
&\le2G^1+\frac{2}{1-\alpha}\sum_{j=2}^{n}G^j p_{n-j}^{(n)}a_{j-2}^{(j)}(t_{j-\frac{1}{2}}-t_{\frac{1}{2}})^\alpha (\tau_{j-\frac{1}{2}})^{-\alpha}\\
&\le c_v\tau_1^\sigma+\frac{c_v}{1-\alpha}\sum_{j=2}^{n}p_{n-j}^{(n)}a_{j-2}^{(j)}(t_{j-\frac{1}{2}}-t_{\frac{1}{2}})^\alpha t_{j-\frac{3}{2}}^{\sigma-2} (\tau_{j-\frac{1}{2}})^{2-\alpha}\\
&\le c_v\Big(\tau_1^\sigma+\frac{1}{1-\alpha}\max_{2\le j\le n}(t_{j-\frac{1}{2}}-t_{\frac{1}{2}})^\alpha t_{j-\frac{3}{2}}^{\sigma-2} (\tau_{j-\frac{1}{2}})^{2-\alpha}\Big).
\end{align*}
The proof ends.
\end{proof}
\begin{remark} Giving a uniform mesh $\tau=\frac{T}{N}$ and $t_k=k\tau,$ then it follows from Corollary \ref{truncation-error2} that
\begin{align*}
\sum_{j=1}^np_{n-j}^{(n)}|R^j|
=& c_v\Big(\tau^\sigma+\frac{1}{2-\beta}\tau^{\min\{\sigma,~ 3-\beta\}}
\max_{2\le k\le n}\Big(k-\frac{3}{2}\Big)^{\beta+\sigma-3}\tau^{\sigma-\min\{\sigma,~ 3-\beta\}}\Big)\\
\le&c_v\Big(\tau^\sigma+t_{k-\frac{3}{2}}^{\sigma-\min\{\sigma,~ 3-\beta\}}\tau^{\min\{\sigma,~ 3-\beta\}}\Big).
\end{align*}
 The above error estimate shows that the convergence order of the difference scheme in time increases along with the improvement of the regularity of the solution for
$\sigma\le 3-\beta.$ Moreover, the convergence order achieves the accuracy of $O(\tau^{3-\beta})$ for $\sigma\in[3-\beta,2).$
\end{remark}

Besides, we consider the truncation error on the graded time mesh $t_k=T(k/N)^\gamma$ with $\gamma>1$
\begin{align*}
\Big(t_{\frac{3}{2}}-t_{\frac{1}{2}}\Big)^{\beta-1} t_{\frac{1}{2}}^{\sigma-2}\tau_{\frac{3}{2}}^{3-\beta}
\le\Big(T(\frac{2}{N})^\gamma\Big)^2\Big(\frac{1}{2}T(\frac{1}{N})^\gamma\Big)^{(\sigma-2)}
=2^{2\gamma-2+\sigma}T^\sigma N^{-\sigma\gamma}
\end{align*}
and it is easy to check the time-step $\tau_{k}\le TN^{-\gamma}\gamma k^{\gamma-1}$. Noticing that $k\le 3(k-2),~k\ge3,$ we have
\begin{align*}
&(t_{k-\frac{1}{2}}-t_{\frac{1}{2}})^{\beta-1} t_{k-\frac{3}{2}}^{\sigma-2}\tau_{k-\frac{1}{2}}^{3-\beta}\\
\le& T^{\beta-1}\Big(\frac{k}{N}\Big)^{(\beta-1)\gamma}
T^{\sigma-2}\Big(\frac{k-2}{N}\Big)^{\gamma(\sigma-2)}\Big(\gamma k^{\gamma-1}TN^{-\gamma}\Big)^{3-\beta}\\
\le&T^\sigma3^{2(\gamma-1)
+\beta-1}(k-2)^{\min\{\gamma\sigma,~3-\beta\}-(3-\beta)}
\Big(\frac{k-2}{N}\Big)^{\gamma\sigma-\min\{\gamma\sigma,~3-\beta\}}\gamma^{3-\beta}N^{-\min\{\gamma\sigma,~3-\beta\}}\\
\le& T^\sigma3^{2(\gamma-1)
+\beta-1}\gamma^{3-\beta}N^{-\min\{\gamma\sigma,~3-\beta\}},~~k\ge 3,
\end{align*}
the above inequalities gives
\begin{align}\label{estimat3}
\sum_{j=1}^np_{n-j}^{(n)}|R^j|
\le c_v N^{-\min\{\gamma\sigma,~ 3-\beta\}},
\end{align}
where $c_v$ is a constant.
\begin{remark}
It follows from the estimate \eqref{estimat3} that the temporal convergence order improves as $\gamma$ increases and the accuracy achieves the optimal $O(N^{\beta-3})$ as taking $\gamma=\max\{1,~(3-\beta)/\sigma\}.$
\end{remark}
 \section{The temporal nonuniform $L1$ type difference scheme}\label{Sec.3}

 In this section, we construct a nonuniform difference scheme for the time fractional diffusion-wave equation.
 Applying the order reduction technique, the problem (\ref{diffusion-equation})-(\ref{boundaryvalue}) can be rewritten by an equivalent equations.
Let $\alpha=\beta-1$  and \begin{align}
v({\rm x},t)= \frac{\partial u}{\partial t}({\rm x},t),\label{0.1}
\end{align}
it reduces \begin{align}
{}_0^CD_t^\beta u({\rm x},t)=&\frac{1}{\Gamma(2-\beta)}\int_0^t
 \frac{\partial^2 u}{\partial s^2}({\rm x},s)\frac{1}{(t-s)^{\beta-1}}ds\nonumber\\
=&\frac{1}{\Gamma(1-\alpha)}\int_0^{t}
 \frac{\partial v}{\partial s}({\rm x},s)\frac{1}{(t-s)^{\alpha}}ds\nonumber\\
 =&{}_0^CD_t^\alpha v({\rm x},t).\label{3.3}
\end{align}
Thus,  Eqs (\ref{diffusion-equation})-(\ref{boundaryvalue}) are equivalent to
\begin{align}
&{}_0^CD_t^\alpha v({\rm x},t)=\Delta u({\rm x},t)+f({\rm x},t),~{\rm x}\in\Omega,~t\in(0,T],\label{equivalent-equation}\\
 & v({\rm x},t)=u_t({\rm x},t),~{\rm x}\in(a,b),~t\in(0,T],\label{v-u}\\
 &u({\rm x},0)=\varphi_1({\rm x}),~v({\rm x},0)=\varphi_2({\rm x}),~{\rm x}\in\bar{\Omega}. \label{v-u-initialvalue}
\end{align}

Let $M$ be a positive integer. Set $\Omega=(0,L)^2$ and $x_i=ih,~y_j=jh$ with the spatial lengths $h=L/M.$
The discrete spatial grid
$\Omega_h:=\big\{{\rm x}_h=(x_i, y_j)~|~1\le i, j\le M-1\big\}$
and $\bar{\Omega}_h:=\big\{{\rm x}_h~|~0\le i, j\le M\big\}.$
Denote $$\mathcal{V}_h:=\big\{v_h=v(\mathrm{x}_h)~|~\mathrm{x}_h\in \bar{\Omega}_h
\;\text{and $v_h$ is $L$-periodic in each direction}\big\}.$$
Given a grid function $v_h\in \mathcal{V}_h,$ introduce the following notations
$\delta_xv_{i+\frac{1}{2},j}=(v_{i+1,j}-v_{ij})/h$
and $\delta^2_xv_{ij}=(\delta_xv_{i+\frac{1}{2},j}-\delta_xv_{i-\frac{1}{2},j})/h.$
Similarly, we define $\delta_yv_{i,j+\frac{1}{2}}$ and $\delta^2_yv_{ij}.$
The discrete Laplacian operator $\Delta_hv_{ij}=\delta^2_xv_{ij}+\delta^2_yv_{ij}$ and
the discrete gradient vector $\nabla_hv_{ij}=(\delta_xv_{i-\frac{1}{2}, j},~\delta_yv_{i, j-\frac{1}{2}})^T$ can be defined.
For any $u,v\in\mathcal{V}_h,$
the inner product and norms are defined by
$$(u,v)=h^2\sum_{{\rm x}_h\in \Omega_h}u_hv_h,~~\|u\|=\sqrt{(u,u)}.$$

Considering the equation (\ref{equivalent-equation}) and (\ref{v-u}) at the point $({\rm x}_h,t_{n-\frac{1}{2}}),$ it yields
\begin{align*}
&\mathcal{D}_{\tau}^\alpha V_h^{n-\frac{1}{2}}=\Delta_hU_h^{n-\frac{1}{2}}+f_h^{n-\frac{1}{2}}+\Upsilon_h^{n},~~{\rm x}_h\in \Omega_h,~~1\le n\le N-1,\\
&V_h^{n-\frac{1}{2}}=\delta_tU_h^{n-\frac{1}{2}}+r_h^{n},~~,{\rm x}_h\in \Omega_h,~~1\le n\le N,\\
&U_h^0=\varphi_1({\rm x}_h),~~V_h^0=\varphi_2({\rm x}_h),~~{\rm x}_h\in \bar{\Omega}_h,
\end{align*}
where $\Upsilon_h^{n}=R_h^{n}+\xi_h^{n},$ and $R_h^{n}$ is the truncation error in time direction, $\xi_h^{n}$ is the truncation error in space direction.
There exists a constant $c_0$ such that
\begin{align}
|r_h^n|\le c_0\tau_n^2,~~|\xi_h^{n}|\le c_0 h^2,~~{\rm x}_h\in \Omega_h,~1\le n\le N-1.\label{truncationerror-time-space}
\end{align}
Omitting the truncation errors, we construct the difference scheme for the fractional diffusion-wave equation as follows
 \begin{align}
&\mathcal{D}_{\tau}^\alpha v_h^{n-\frac{1}{2}}=\Delta_hu_h^{n-\frac{1}{2}}+f_h^{n-\frac{1}{2}},~~{\rm x}_h\in\Omega_h,~~1\le n\le N-1,\label{discrete-scheme}\\
&v_h^{n-\frac{1}{2}}=\delta_tu_h^{n-\frac{1}{2}},~~{\rm x}_h\in\Omega_h,~~1\le n\le N,\label{discrete-v-u}\\
&u_h^0=\varphi_1({\rm x}_h),~~v_h^0=\varphi_2({\rm x}_h),~~{\rm x}_h\in\bar{\Omega}_h.\label{discrete-u-v-initialvalue}
\end{align}
\end{proof}
\section{The error estimate of the difference scheme}\label{Sec.4}

Applying the important discrete tool $p_{n-k}^{(n)},$ we present the convergence analysis of
the nonuniform difference scheme (\ref{discrete-scheme})-(\ref{discrete-u-v-initialvalue}).
Denote $$e_h^n=U_h^n-u_h^n,~~\rho_h^n=V_h^n-v_h^n,~~{\rm x}_h\in \bar{\Omega}_h,~0\le n\le N,$$
the error equation gives as follows
\begin{align}
&\mathcal{D}_{\tau}^\alpha \rho_h^{n-\frac{1}{2}}=\Delta_he_h^{n-\frac{1}{2}}+\Upsilon_h^{n},~~{\rm x}_h\in\Omega_h,~~1\le n\le N,\label{error-scheme}\\
&\rho_h^{n-\frac{1}{2}}=\delta_te_h^{n-\frac{1}{2}}+r_h^{n},~~{\rm x}_h\in\Omega_h,~~1\le n\le N,\label{error-v-u}\\
&e_h^0=0,~~\rho_h^0=0,~~{\rm x}_h\in\bar{\Omega}_h.\label{error-u-v-initialvalue}
\end{align}

\begin{theorem}\label{convergence-th}
Suppose the problem (\ref{diffusion-equation}) has a unique smooth solution and $u_h^n\in\mathcal{V}_h$ is the solution of the difference scheme  (\ref{discrete-scheme})-(\ref{discrete-u-v-initialvalue}). The proposed scheme (\ref{discrete-scheme})-(\ref{discrete-u-v-initialvalue}) is convergent in $L^2$ norm,
$$ \|e^n\|\le c_v\Big(\max_{1\le k\le n}\sum_{j=1}^{k} p_{k-j}^{(k)}a_0^{(j)}\int_{t_{j-\frac{3}{2}}}^{t_{j-\frac{1}{2}}}\Big(t-t_{j-\frac{3}{2}}\Big)|\partial_{tt}U| dt+t_{n-\frac{1}{2}}^{\beta-1} h^2\Big).$$
\end{theorem}
\begin{proof} Multiplying (\ref{error-scheme}) by {$p_{n-k}^{(n)}$} and summing up $k$ from 1 to $n$, it yields
\begin{align}
\sum_{k=1}^{n}{p_{n-k}^{(n)}}\sum_{j=1}^ka_{k-j}^{(k)}\nabla_\tau\rho_h^{j-\frac{1}{2}}= \sum_{k=1}^{n}{p_{n-k}^{(n)}}\Delta_h e_h^{k-\frac{1}{2}}
+\sum_{k=1}^{n}{p_{n-k}^{(n)}}\Upsilon_h^{k}.\label{theta-equation}
\end{align}
Exchanging the summation order, we have
\begin{align*}
 \sum_{k=1}^{n}{p_{n-k}^{(n)}}\sum_{j=1}^ka_{k-j}^{(k)}\nabla_\tau\rho_h^{j-\frac{1}{2}}
 =\sum_{j=1}^n\nabla_\tau\rho_h^{j-\frac{1}{2}}\sum_{k=j}^n{p_{n-k}^{(n)}}a_{k-j}^{(k)}={\rho_h^{n-\frac{1}{2}}}.
\end{align*}
Consequently, it yields
\begin{align}
\rho_h^{n-\frac{1}{2}}=\sum_{k=1}^np_{n-k}^{(n)}\Delta_he_h^{k-\frac{1}{2}}
+\sum_{k=1}^np_{n-k}^{(n)}\Upsilon_h^{k}.\label{rho-e}
\end{align}
Taking the inner product of (\ref{rho-e}) with $e^{n-\frac{1}{2}},$ we have
\begin{align}
\langle\rho^{n-\frac{1}{2}}, e^{n-\frac{1}{2}}\rangle=\sum_{k=1}^np_{n-k}^{(n)}\langle\Delta_he^{k-\frac{1}{2}}, e^{n-\frac{1}{2}}\rangle
+\sum_{k=1}^np_{n-k}^{(n)}\langle\Upsilon^{k}, e^{n-\frac{1}{2}}\rangle,\label{rho-e-inner1}
\end{align}
and the inner product of (\ref{error-v-u}) with $e^{n-\frac{1}{2}}$, one arrives at
\begin{align}
\langle\rho^{n-\frac{1}{2}}, e^{n-\frac{1}{2}}\rangle=\langle\delta_t e^{n-\frac{1}{2}}, e^{n-\frac{1}{2}}\rangle+\langle r^n, e^{n-\frac{1}{2}}\rangle.\label{rho-e-inner2}
\end{align}
Substituting (\ref{rho-e-inner2}) into (\ref{rho-e-inner1}), and summing $k$ from 1 to $n$, we get
\begin{align*}
\sum_{k=1}^n\langle\nabla_\tau e^{k}, e^{k-\frac{1}{2}}\rangle=&\sum_{k=1}^n\tau_k\sum_{l=1}^kp_{k-l}^{(k)}\langle\Delta_he^{l-\frac{1}{2}}, e^{k-\frac{1}{2}}\rangle
+\sum_{k=1}^n\tau_k\sum_{l=1}^kp_{k-l}^{(k)}\langle\Upsilon^{l}, e^{k-\frac{1}{2}}\rangle\\
&-\sum_{k=1}^n\tau_k\langle r^k, e^{k-\frac{1}{2}}\rangle.
\end{align*}
It is easy to obtain
\begin{align*}
\sum_{k=1}^n\langle \nabla_\tau e^{k}, e^{k-\frac{1}{2}}\rangle=\frac{1}{2}(\|e^n\|^2-\|e^0\|^2).
\end{align*}
With the help of the positive definiteness of $p_{k-l}^{(k)}$, we have
$$\sum_{k=1}^n\tau_k\sum_{l=1}^kp_{k-l}^{(k)}\langle\Delta_he^{l-\frac{1}{2}}, e^{k-\frac{1}{2}}\rangle
=-\sum_{k=1}^n\tau_k\sum_{l=1}^kp_{k-l}^{(k)}\langle \nabla_he^{l-\frac{1}{2}}, \nabla_he^{k-\frac{1}{2}}\rangle\le 0.$$
Then, it follows that
\begin{align}
\|e^n\|^2\le 2\sum_{k=1}^n\tau_k\sum_{l=1}^kp_{k-l}^{(k)}\|\Upsilon^{l}\|\cdot\|e^{k-\frac{1}{2}}\|
+2\sum_{k=1}^n\tau_k\|r^k\|\cdot\|e^{k-\frac{1}{2}}\|.
\end{align}

Choosing some integer $n_0 (0\le n_0\le n)$ such that $\|e^{n_0}\|=\max\limits_{0\le k\le n}\|e^k\|$ and then taking $n=n_0$ in the above inequality,
it yields
\begin{align*}
\|e^{n_0}\|^2\le 2\sum_{k=1}^n\tau_k\Big(\sum_{l=1}^kp_{k-l}^{(k)}\|\Upsilon^{l}\|\Big)\cdot\|e^{n_0}\|+2\sum_{k=1}^n\tau_k\|r^k\|\cdot\|e^{n_0}\|.
\end{align*}
Consequently, by virtue of Lemma \ref{DCC-property}, Lemma \ref{truncationerror} and noticing $\alpha=\beta-1$, it yields
\begin{align*}
\|e^n\|\le\|e^{n_0}\|\le&2\sum_{k=1}^n\tau_k\Big(\sum_{l=1}^kp_{k-l}^{(k)}\|\Upsilon^{l}\|\Big)
+2\sum_{k=1}^n\tau_k\|r^k\|\\
\le&2c_0t_n\tau_n^2
+2t_n\max_{1\le k\le n}\sum_{l=1}^kp_{k-l}^{(k)}(\|R^{l}\|+\|\xi^l\|)\\
\le&c_v\Big(\max_{1\le k\le n}\sum_{j=1}^{k} p_{k-j}^{(k)}a_0^{(j)}\int_{t_{j-\frac{3}{2}}}^{t_{j-\frac{1}{2}}}\Big(t-t_{j-\frac{3}{2}}\Big)|\partial_{tt}U| dt+t_{n-\frac{1}{2}}^{\beta-1} h^2+t_n\tau_n^2\Big).
\end{align*}
Thus we obtain the desired result. The proof ends.
\end{proof}

\begin{remark} The present analysis takes advantage of the non-negative (Lemma \ref{DCC-property}) and the positive definiteness (Lemma \ref{DCCposi}) of the DCC kernels defined by \eqref{DCC equlity}. The current framework is extendable to the nonlinear diffusion-wave problems if the numerical solution is bounded in certain discrete norm (such as the $H^1$ norm). This issue is interesting and will be studied in the further study.

To improve the time accuracy, one can employ some high-order approximations for the reduced equation (\ref{equivalent-equation}). Typically,  we can apply the fractional BDF2 formula\cite{Liao-Liu-Zhao, Quan-Wu} to build the following second-order variable-step scheme
\begin{align*}
&D_{\tau}^\alpha v_h^{n}=\Delta_hu_h^{n}+f_h^{n},~~{\rm x}_h\in\Omega_h,~~1\le n\le N-1,\\
&v_h^{n}=\frac{1+2r_n}{\tau_n(1+r_n)}\nabla_\tau u_h^n-\frac{r_n^2}{\tau_n(1+r_n)}\nabla_\tau u_h^{n-1},~~{\rm x}_h\in\Omega_h,~~1\le n\le N,\\
&u_h^0=\varphi_1({\rm x}_h),~~v_h^0=\varphi_2({\rm x}_h),~~{\rm x}_h\in\bar{\Omega}_h,
\end{align*}
where the fractional BDF2 formula is given as $D_{\tau}^\alpha v_h^{n}=\sum\limits_{k=1}^{n}B_{n-k}^{(n)}\nabla_\tau v_h^k$ and the coefficients are written as
\begin{align*}
&B_0^{(1)}=a_0^{(1)},~~B_0^{(n)}=a_0^{(n)}+\frac{r_n^2\varpi_0^{(n)}+\varpi_1^{(n)}}{r_n(1+r_n)},~n\ge2,\\
&B_1^{(n)}=a_1^{(n)}-\frac{r_n^2\varpi_0^{(n)}+\varpi_1^{(n)}}{1+r_n}+\frac{\varpi_2^{(n)}}{r_{n-1}(1+r_{n-1})},~n\ge 2,\\
&B_{n-k}^{(n)}=a_{n-k}^{(n)}-\frac{\varpi_{n-k}^{(n)}}{1+r_{k+1}}+\frac{\varpi_{n-k+1}^{(n)}}{r_{k}(1+r_{k})},~2\le k\le n-1,~n\ge3,\\
&B_{n-1}^{(n)}=a_{n-1}^{(n)}-\frac{\varpi_{n-1}^{(n)}}{1+r_2},~n\ge 2,
\end{align*} in which the parameters $a_{n-k}^{(n)}$ and  $\varpi_{n-k}^{(n)}$ are shown in the integral forms $$a_{n-k}^{(n)}=\frac{1}{\tau_k}\int_{t_{k-1}}^{t_k}\omega_{1-\alpha}(t_n-s) ds,~1\le k\le n$$ and
$$\varpi_{n-k}^{(n)}=\frac{1}{\tau_k}\int_{t_{k-1}}^{t_k}\frac{2s-t_k-t_{k-1}}{\tau_k}\omega_{1-\alpha}(t_n-s) ds,~1\le k\le n.$$
\end{remark}
Another second-order nonuniform scheme can also be constructed by using the well-known  $L2-1\sigma$ formula\cite{Liao-Tang-Zhou}. The associated stability and convergence analysis on nonuniform time meshes are also very interesting, however, these tasks are rather challenging due to  the nonuniform setting and the inhomogeneity of the discrete coefficients $B_{n-k}^{(n)}$.  These issues are planned to explore and will be presented in separate reports.

\section{Numerical Experiment}\label{Sec.5}

In this section, some numerical examples are demonstrated for the fractional diffusion-wave equation to verify the efficiency of the difference scheme (\ref{discrete-scheme})-(\ref{discrete-u-v-initialvalue}).

\begin{example}\label{ex5.1}
Take $\Omega=(0,2\pi)^2,$ consider the problem (\ref{diffusion-equation}) with the source term
\begin{align*}
f({\rm x},t)=\Big(2t^{\sigma+1}+\frac{\Gamma(2+\sigma)}{\Gamma(\sigma+2-\beta)}t^{\sigma+1-\beta}\Big)\sin x\sin y.
\end{align*}
The problem has an exact solution
$$u({\rm x},t)=t^{\sigma+1}\sin x\sin y.$$
\end{example}
Denote $$ e(N)=\|U^N-u^N\|_\infty,~~{\rm{Order_{\tau}=log_{2}}}(e(N)/e(2N))$$ for the error and the convergence order.

Example \ref{ex5.1} is presented to measure the accuracy of the difference scheme (\ref{discrete-scheme})-(\ref{discrete-u-v-initialvalue}) in time direction. We take the graded mesh $t_n=T(n/N)^\gamma.$ The spatial grid
node is fixed to $M=1000.$ Table \ref{accuracy-table1} and Table \ref{accuracy-table2} demonstrate the $L^2$ norm errors $e(N)$ and the convergence orders of the difference scheme
(\ref{discrete-scheme})-(\ref{discrete-u-v-initialvalue}) in time direction with the regularity parameter $\sigma=\beta-1$ and $\sigma=\beta/2$ for $\gamma=1, 2, 3,$ respectively. From two tables, we observe that
the difference scheme (\ref{discrete-scheme})-(\ref{discrete-u-v-initialvalue}) reaches the accuracy of $O(N^{-\min\{\gamma\sigma, ~3-\beta\}})$ in time which is in accord with our theoretical result.

\begin{table}[htbp]
\begin{center}
\renewcommand{\arraystretch}{1.12}
\def\temptablewidth{0.9\textwidth}
\rule{\temptablewidth}{1pt}
{\footnotesize
\begin{tabular*}{\temptablewidth}{@{\extracolsep{\fill}}c|cccccccc}
&&\multicolumn{2}{c} {$\gamma=1$}~~&\multicolumn{2}{c} {$\gamma=2$}~~&\multicolumn{2}{c} {$\gamma=3$}\\
\cline{3-4}\cline{5-6} \cline{7-8}
$\beta$~&$N$&$e(N)$& ${\rm order}_\tau$ &$e(N)$& ${\rm order}_\tau$&$e(N)$&${\rm order}_\tau$\\
\hline
\multirowcell{8}&$ 40$~~~ &  2.79e-01   &   -   &  2.01e-01  &   -     & 1.39e-01  &    -  \\
                &$ 80$~~~ &  2.66e-01   &  0.07 &  1.75e-01  &  0.20   & 1.13e-01  &  0.30 \\
         $1.1~~$&$160$~~~ &  2.51e-01   &  0.09 &  1.52e-01  &  0.20   & 9.18e-02  &  0.30 \\
                &$320$~~~ &  2.35e-01   &  0.09 &  1.33e-01  &  0.20   & 7.45e-02  &  0.30  \\
\hline
&$\min\{\gamma\sigma, 3-\beta\}$&&0.1&&0.2&&0.3\\
\hline
\multirowcell{8}&$40 $~~~ &  5.56e-02   &  -    &  9.19e-03   &   -     & 9.19e-05  &    - \\
                &$80 $~~~ &  4.28e-02   &  0.38 &  4.93e-03   &  0.90   & 2.69e-05  &  1.72 \\
         $1.5~~$&$160$~~~ &  3.19e-02   &  0.42 &  2.57e-03   &  0.94   & 8.60e-06  &  1.70 \\
                &$320$~~~ &  2.34e-02   &  0.45 &  1.33e-03   &  0.96   & 2.81e-06  &  1.61 \\
\hline
&$\min\{\gamma\sigma, 3-\beta\}$&&0.5&&1.0&&1.5\\
\hline
\multirowcell{8}&$40 $~~~ &  8.45e-04   &  -    & 2.94e-03    &  -      & 3.64e-03  &    -  \\
                &$80 $~~~ &  9.82e-04   &  -0.22& 1.43e-03    &  1.04   & 1.66e-03  &  1.13  \\
         $1.9~~$&$160$~~~ &  7.99e-04   &  0.30 & 6.86e-04    &  1.06   & 7.67e-04  &  1.12  \\
                &$320$~~~ &  5.70e-04   &  0.49 &3.25e-04     &  1.08   & 3.55e-04  &  1.11   \\
\hline
&$\min\{\gamma\sigma, 3-\beta\}$&&0.9&&1.1&&1.1\\
\hline
\end{tabular*}}
\rule{\temptablewidth}{1pt}
\end{center}
\tabcolsep 0pt \caption{$L_2$ errors and convergence orders of the difference scheme in time for $\sigma=\beta-1.$}\label{accuracy-table1}
\end{table}

\begin{example}
Consider the following fractional Klein-Gordon equation
$${}_0^CD_t^\beta u-\varepsilon^2\Delta u+u^3=f,~~{\rm x}_h\in(0,2\pi)^2,~t\in [0,1],$$
subjected to the periodic boundary condition, where $f$ is a source term such that the exact solution is
$$u({\rm x}_h,t)=t^\beta \sin x\sin y.$$
\end{example}

\begin{table}[htbp]
\begin{center}
\renewcommand{\arraystretch}{1.12}
\def\temptablewidth{0.9\textwidth}
\rule{\temptablewidth}{1pt}
{\footnotesize
\begin{tabular*}{\temptablewidth}{@{\extracolsep{\fill}}c|cccccccc}
&&\multicolumn{2}{c} {$\gamma=1$}~~&\multicolumn{2}{c} {$\gamma=2$}~~&\multicolumn{2}{c} {$\gamma=3$}\\
\cline{3-4}\cline{5-6}\cline{7-8}
$\beta$~&$N$&$e(N)$& ${\rm order}_\tau$ & $e(N)$& ${\rm order}_\tau$ &$e(N)$&${\rm order}_\tau$\\
\hline
\multirowcell{8}&$ 40$~~~ &  5.75e-02   &   -    &  7.71e-03   &   -     & 7.10e-04  &    -   \\
                &$ 80$~~~ &  4.02e-02   & 0.52   &  3.64e-03   &  1.08   & 2.43e-04  &  1.55  \\
         $1.1~~$&$160$~~~ &  2.78e-02   & 0.53   &  1.71e-03   &  1.09   & 8.30e-05  &  1.55   \\
                &$320$~~~ &  1.91e-02   & 0.54   &  8.02e-04   &  1.09   & 2.89e-05  &  1.52   \\
\hline
&$\min\{\gamma\sigma, 3-\beta\}$&&0.55&&1.1&&1.65\\
\hline
\multirowcell{8}&$40 $~~~ &  2.54e-02   &   -    &  9.88e-04   &   -     & 1.18e-03  &    -   \\
                &$80 $~~~ &  1.63e-02   & 0.64   &  3.71e-04   &  1.41   & 4.06e-04  &  1.54    \\
         $1.5~~$&$160$~~~ &  1.02e-02   & 0.68   &  1.37e-04   &  1.44   & 1.39e-04  &  1.54     \\
                &$320$~~~ &  6.24e-03   & 0.71   &  5.01e-05   &  1.45   & 4.75e-05  &  1.55     \\
\hline
&$\min\{\gamma\sigma, 3-\beta\}$&&0.75&&1.5&&1.5\\
\hline
\multirowcell{8}&$40 $~~~ &  2.86e-03   &   -    & 1.45e-03    &  -      & 1.98e-03  &    -    \\
                &$80 $~~~ &  1.95e-03   & 0.55   & 6.93e-04    &  1.06   & 8.67e-04  &  1.19    \\
         $1.9~~$&$160$~~~ &  1.23e-03   & 0.66   & 3.31e-04    &  1.07   & 3.89e-04  &  1.16     \\
                &$320$~~~ &  7.46e-04   & 0.72   & 1.57e-04    &  1.08   & 1.77e-04  &  1.13    \\
\hline
&$\min\{\gamma\sigma, 3-\beta\}$&&0.95&&1.1&&1.1\\
\hline
\end{tabular*}}
\rule{\temptablewidth}{1pt}
\end{center}
\tabcolsep 0pt \caption{$L_2$ errors and convergence orders of the difference scheme in time for $\sigma=\beta/2.$}\label{accuracy-table2}
\end{table}

The purpose of this test is to verify the convergence rate of the scheme in time for the fractional Klein-Gordon equation. We also take the graded mesh $t_n=T(n/N)^\gamma$ on the interval $[0,1].$ The spatial domain $(0,2\pi)^2$ is
discretized with $1000^2$ grids. Table \ref{accuracy-order-table} lists the $L^2$ norm errors and the convergence orders for the different fractional orders $\beta=1.2, 1.5, 1.9$ and grading parameters $\gamma=2,3, 5$, and the regularity parameter $\sigma=\beta-1$.
From Table \ref{accuracy-order-table}, we observe that the difference scheme achieves the expected temporal accuracy $O(\tau^{\min\{\gamma\sigma,~ 3-\beta\}}).$

\begin{table}[htbp]
\begin{center}
\renewcommand{\arraystretch}{1.12}
\def\temptablewidth{0.9\textwidth}
\rule{\temptablewidth}{1pt}
{\footnotesize
\begin{tabular*}{\temptablewidth}{@{\extracolsep{\fill}}c|ccccccc}
&&\multicolumn{2}{c} {$\gamma=2$}~~&\multicolumn{2}{c} {$\gamma=3$}~~&\multicolumn{2}{c} {$\gamma=5$}\\
\cline{3-4}\cline{5-6}\cline{7-8}
$\beta$~&$N$&$e(N)$& ${\rm order}_\tau$ &$e(N)$& ${\rm order}_\tau$&$e(N)$& ${\rm order}_\tau$\\
\hline
\multirowcell{8}&$ 40$~~~ & $1.46{\rm e}-1$ & $-$   & $1.00{\rm e}-1$ & $ -  $& $4.58{\rm e}-2$ & $-$\\
                &$ 80$~~~ & $1.29{\rm e}-1$ & $0.17$& $8.35{\rm e}-2$ & $0.26$& $3.35{\rm e}-2$ & $0.45$\\
         $1.1~~$&$160$~~~ & $1.14{\rm e}-1$ & $0.18$& $6.89{\rm e}-2$ & $0.28$& $2.43{\rm e}-2$ & $0.46$\\
                &$320$~~~ & $1.00{\rm e}-1$ & $0.19$& $5.65{\rm e}-2$ & $0.29$& $1.76{\rm e}-2$ & $0.46$\\
\hline
&$\min\{\gamma\sigma, 3-\beta\}$&&0.2&&0.3&&0.5\\
\hline
\multirowcell{8}&$40$~~~ & $6.64{\rm e}-3$ & $-$   & $1.01{\rm e}-2$ & $-$   & $1.89{\rm e}-2$ & $-$\\
                &$80$~~~ & $3.66{\rm e}-3$ & $0.86$& $4.88{\rm e}-3$ & $1.05$& $8.82{\rm e}-3$ & $1.10$\\
         $1.5~~$&$160$~~~& $1.95{\rm e}-3$ & $0.91$& $2.40{\rm e}-3$ & $1.02$& $4.20{\rm e}-3$ & $1.07$\\
                &$320$~~~& $1.01{\rm e}-3$ & $0.94$& $1.19{\rm e}-3$ & $1.01$& $2.04{\rm e}-3$ & 1.05\\
\hline
&$\min\{\gamma\sigma, 3-\beta\}$&&1.0&&1.5&&1.5\\
\hline
\multirowcell{8}&$40$~~~ & $6.66{\rm e}-3$ & $ -  $& $9.37{\rm e}-3$ & $  - $& $1.56{\rm e}-2$ & $-$\\
                &$80$~~~ & $3.17{\rm e}-3$ & $1.07$& $4.28{\rm e}-3$ & $1.13$& $6.85{\rm e}-3$ & $1.19$\\
         $1.9~~$&$160$~~~ & $1.53{\rm e}-3$ & $1.05$& $2.02{\rm e}-3$ & $1.09$& $3.15{\rm e}-3$ & $1.12$\\
                &$320$~~~ & $7.41{\rm e}-4$ & $1.04$& $9.65{\rm e}-4$ & $1.06$& $1.49{\rm e}-3$ & $1.08$\\
\hline
&$\min\{\gamma\sigma, 3-\beta\}$&&1.1&&1.1&&1.1\\
\hline
\end{tabular*}}
\rule{\temptablewidth}{1pt}
\end{center}
\tabcolsep 0pt \caption{$L_2$ errors and convergence orders of the difference scheme in time}\label{accuracy-order-table}
\end{table}

\section{Conclusion}\label{Sec.6}
In this paper, by using the order reduction technique, we presented a nonuniform $L1$ type difference scheme for the fractional diffusion-wave equation at the half grid points based on the piecewise linear interpolation. By virtue of the discrete
DCC kernels $p_{n-k}^{(n)}$, the unconditional $L^2$ norm convergence analysis is obtained for the proposed difference scheme. We employed the scheme
on the graded mesh to perform some numerical tests. These tests suggested that the nonuniform difference scheme (\ref{discrete-scheme})-(\ref{discrete-u-v-initialvalue}) can achieve $\min\{\gamma\sigma, 3-\beta\}$ order accuracy which confirmed the theoretical
result. In the future work, we will study the nonuniform numerical scheme for the time fractional nonlinear equation with the Caputo derivative $\beta\in (1,2).$
\section*{Acknowledgements}
We would like to acknowledge support by the State Key Program of National Natural Science Foundation of China (No. 11931003, 61833005), the National Natural Science Foundation of China (No. 41974133, 12126325, 11701081, 11701229, 11861060), ZhiShan Youth Scholar Program of SEU, China Postdoctoral Science Foundation (No. 2019M651634), High-level Scientific Research foundation for the introduction of talent of Nanjing Institute of Technology (No. YKL201856). 

The authors appreciate the anonymous referee for the valuable comments and suggestions.

\end{document}